\documentclass[11pt,a4paper]{amsart}
\usepackage[margin=1.1in]{geometry}
\usepackage{amsmath,amssymb,amsthm,mathrsfs}
\usepackage[colorlinks=true,linkcolor=blue,citecolor=blue]{hyperref}
\usepackage{enumitem}
\usepackage{tikz}
\usepackage{amsfonts,amsmath,oldgerm,amscd}
\usepackage{quiver}
\usepackage{tikz-cd}

\theoremstyle{plain}
\newtheorem{thm}{Theorem}[section]
\newtheorem{lem}[thm]{Lemma}
\newtheorem{prop}[thm]{Proposition}
\newtheorem{cor}[thm]{Corollary}

\theoremstyle{definition}
\newtheorem{defn}[thm]{Definition}
\newtheorem{ex}[thm]{Example}

\theoremstyle{remark}

\usepackage{lipsum}
\date{}

\title[GW groups and cancellation of quadratic spaces]{INVARIANCE OF GROTHENDIECK-WITT GROUPS UNDER SUBINTEGRAL EXTENSIONS AND CANCELLATION OF QUADRATIC SPACES}
\author{Jebasingh R}
\address{Department of Mathematics, Indian Insitute of Technology Madras, Chennai, Tamil Nadu 600036}
\email{jebasinghr3@gmail.com, ma25r014@smail.iitm.ac.in}
\begin{document}
\maketitle
\begin{abstract}
    Motivated by the work of Ischebeck on the behaviour of algebraic $K$-theoretic functors under subintegral extensions, we investigate the analogous question in Hermitian $K$-theory. We prove that Grothendieck-Witt groups in degree zero are invariant under  subintegral extensions. We also investigate the cancellation problem for quadratic spaces over geometric subrings of polynomial rings and establish a cancellation theorem for quadratic spaces of sufficiently large Witt index.
\end{abstract}
\section{Introduction}
Let $R \subset S$ be a finite integral extension of rings. For a point $x \in \operatorname{Spec}(R)$, let
$x_1,\ldots,x_n \in \operatorname{Spec}(S)$
be the points lying above $x$. Traverso \cite[p. 588]{traverso1970seminormality} constructs an intermediate ring
$R \subset R' \subset S$
with the property that there exists a unique point $x' \in \operatorname{Spec}(R')$ lying above $x$, and the induced extension of residue fields $\kappa(x)\rightarrow \kappa(x')$
is an isomorphism. Moreover, $R'$ is the largest subring of $S$ containing $R$ with this property.
The seminormalisation of $R$ in $S$, denoted by $^{+}_{S}R$, is defined to be the largest subring
$R \subset  ^+_SR \subset S$
such that, for every point $x \in \operatorname{Spec}(R)$, there exists a unique point of $\operatorname{Spec}(^{+}_{S}R)$ lying above $x$, and the induced map on residue fields is an isomorphism.

For a reduced ring $R$, Traverso showed that when the normalization of $R$ in $Q(R)$ is finite over $R$, $R$ is seminormal in $Q(R)$ ($R=^+_{Q(R)}R$, where $Q(R)$ is the total ring of fractions) is equivalent to the $\operatorname{Pic}(R)\simeq \operatorname{Pic}(R[X_1,\cdots, X_n])$ for any natural number $n$.
Subsequently, Swan \cite{swan1980seminormality}, motivated by a criterion of Hamann \cite[Proposition 2.10]{hamann1975r}, defined a ring $R$ to be seminormal if its reduced and whenever $b,c\in R$ satisfy $b^3=c^2$, there is an $a\in R$ with $a^2=b$ and $a^3=c$. The definition of Swan is equivalent to that of Traverso when $Q(R)$ is product of fields. He showed that for any commutative ring the invariance of Picard group under polynomial extension is equivalent to its reduced ring being seminormal generalising the result of Traverso.

To study seminormality, Swan introduced the more general notion of a subintegral extension (see  \ref{subintdefn}). The study of algebraic $K$-theoretical functors under subintegral extensions was subsequently undertaken by Ischebeck \cite{ischebeck1989subintegral}. He showed that finite subintegral extensions of Noetherian rings induce isomorphisms on the $G$-groups $G_i(S)\xrightarrow{\sim} G_i(R)$, for $i=0,1$
and the Chow groups $CH_r(S)\xrightarrow{\sim} CH_r(R)$
for all $r$. Furthermore, he established several results concerning the behaviour of $K$-groups and Nil $K$-groups under subintegral extension. These results motivate the study of analogous questions for Hermitian $K$-theory. In this note we initiate such a study and prove the following.
\begin{thm}
     Let $R\subset S$ be a finite subintegral extension. Let $1/2\in R$. Then the induced natural map $GW^{[n]}_0(R)\xrightarrow[]{}GW^{[n]}_0(S)$ is an isomorphism for $n=0,2$.
\end{thm}

We give an example where the isomorphism doesn't hold for odd shift.

\medskip
In Section \ref{section2}, we study the cancellation problem for quadratic spaces. Let $q$ be a quadratic $R$-space. We say that $q$ is cancellative if, whenever $q \perp q'' \simeq q' \perp q''$
for quadratic $R$-spaces $q'$ and $q''$, it follows that $q \simeq q'$.
The cancellation problem for projective modules has a long history in algebraic $K$-theory. Bass \cite{bass1964k} proved a fundamental cancellation theorem for projective modules. Amit Roy \cite{roy1968cancellation} established an analogue of Bass cancellation for quadratic spaces, showing that quadratic spaces with sufficiently large Witt index are cancellative. Later, Parimala Raman \cite{parimala1984cancellation} proved the cancellation of quadratic spaces with large Witt index over polynomial rings.
Projective modules over a class of subrings of polynomial rings known as geometric subrings (see \ref{geometricsubring}) have recently been studied in \cite{banerjee2024subrings}. Motivated, we prove the following result.

\begin{thm}
    Let $R$ be a ring of dimension $d\geq 1$. Let $A$ be a geometric subring of $R[T]$. Let  $q$ be a quadratic space over $A$ of Witt index $\geq d+1$. Then $q$ is cancellative.  
\end{thm}

\section{Grothendieck-Witt groups and subintegral extensions}\label{section1}

We begin by recalling the following from \cite[Lemma 2.1]{swan1980seminormality}.
\begin{prop}\label{subintegraleq}
Let $R\subset S$ be an extension of rings. Then the following statements are equivalent. 
\begin{enumerate}
    \item The extension $R\subset S$ is integral extension, the induced map $i:\operatorname{Spec}(S)\longrightarrow \operatorname{Spec}(R)$ is bijective, and for each
$\mathfrak{q}\in \operatorname{Spec}(S)$  the
induced field extension of residue fields $R_{\mathfrak{p}}/ \mathfrak{p} R_{\mathfrak{p}} \subset S_{\mathfrak{q}}/{\mathfrak{q} S_{\mathfrak{q}}}$ is trivial,
where $\mathfrak{p}=i(\mathfrak{q})$.
\item The extension $R\subset S$ is an integral extension. For any field $k$, every $k\text{-}$valued point of $\operatorname{Spec}(R)$ can be lifted uniquely to a $k\text{-}$valued point of $\operatorname{Spec}(S)$. 

\[\begin{tikzcd}
	& {\operatorname{Spec}(S)} \\
	{\operatorname{Spec}(k)} & {\operatorname{Spec}(R)}
	\arrow[from=1-2, to=2-2]
	\arrow["{\exists!}", dashed, from=2-1, to=1-2]
	\arrow[from=2-1, to=2-2]
\end{tikzcd}\]
\end{enumerate}
\end{prop}

\begin{defn}\label{subintdefn}
    An extension $R\subset S$ of rings is called subintegral if any of the equivalent statements of  \ref{subintegraleq} is satisfied.
\end{defn}
 The following result follows from  \ref{subintegraleq} and \cite[Lemma 4.4]{swan1980seminormality}. 
\begin{prop}\label{elementarysub}
  An extension of the form $R\subset S=R[b]$ with $b^2,b^3\in R$ is subintegral.
\end{prop}
We shall refer to subintegral extensions of the type appearing in \ref{elementarysub} as elementary subintegral extensions. The reason for this terminology is explained by the following proposition and the proof follows from \cite[Lemma 2.3, Lemma 2.6]{swan1980seminormality}.
\begin{prop}\label{etaelementary}
    Let $R\subset S$ be a  subintegral. Then $S$ is the filtered union of subrings which can be obtained from $R$ by a finite chain of elementarily subintegral extensions .
\end{prop}
\begin{ex}
\begin{enumerate}
The following are examples of subintegral extensions.
    \item Whitney Umbrella- $R[X_1,\cdots, X_{d-1}, X_1X_d, 
\cdots ,X_{d-1}X_d, X_d^2,X_d^3] \subset R[X_1,\cdots, X_d]$. 
    \item Let $S$ be any Noetherian normal ring of finite type over a field $k=\overline{k}$ whose singular locus in $\operatorname{Spec}R$ has dimension $>0$. Let $\mathfrak{m}$ be any maximal ideal in the singular locus. Consider the subring $R\subset S$ defined by
$R = \{f \in S \mid \overline{f}\in S/\mathfrak{m}^2 \text{ is an element of k} \subset S/\mathfrak{m}^2\}$. Then $R\subset S$ is a subintegral extension with radical of conductor ideal $\mathfrak{m}$. 

\end{enumerate}
\end{ex}
For background on Grothendieck-Witt groups, we refer the reader to \cite{schlichting2017hermitian}. We write $GW^{[n]}_i(R)$ for $GW^{[n]}_i(\mathrm{sPerf}^R(\operatorname{Spec}(R)))$ in the notation of \cite[Definition 9.1]{schlichting2017hermitian}. 

We now recall the following result of C. T. C. Wall \cite[Lemma 5]{wall1973classification}. A proof may be found in \cite[Proposition 2.3.7]{calmes2026hermitian}.
\begin{prop}\label{Ltheoryiso}
  Let $R$ be a ring, complete in the $I$-adic topology for an ideal $I$ of $R$. Then the canonical
map $L^{[n]}(R)\rightarrow L^{[n]}(R/I)$ is an equivalence.
\end{prop}
The following observation describes the behaviour of Grothendieck–Witt groups under nilpotent extensions. Let $R_{red}$ denote the reduced ring of $R$. 
\begin{prop}\label{reducedgw1}
    Let $R$ be a ring with $1/2\in R$. Then the induced map $GW^{[n]}_1(R)\rightarrow GW^{[n]}_1{(R_{red})}$ is an isomorphism for all $n\in \mathbb{N}$.
\end{prop}
 
\begin{proof}
 We recall that the $C_2$-action on $K(R)$ described in \cite[Section 7]{schlichting2017hermitian} induces an action on $K_0(R)$ and $K_1(R)$. The induced action on $K_0(R)$ sends the class of a finitely generated projective module $P$ to the class of its dual module $P^*$, while the induced action on $K_1(R)$ sends the class of an invertible matrix $M$ to the class of $(M^t)^{-1}$.
  We get the following commutative diagram with exact rows from \cite[Theorem 7.6]{schlichting2017hermitian}.
    \[
    \scalebox{0.8}{
    \begin{tikzcd}[ampersand replacement=\&,cramped]
	{L^{[n]}_2(R)} \& {\pi_1(K(R)_{hC_2})} \& {GW^{[n]}_1(R)} \& {L^{[n]}_1(R)} \& {\pi_0(K(R)_{hC_2})} \\
	{L^{[n]}_2(R_{red})} \& {\pi_1(K(R_{red})_{hC_2})} \& {GW^{[n]}_1(R_{red})} \& {L^{[n]}_1(R_{red})} \& {\pi_0(K(R_{red})_{hC_2})}
	\arrow[from=1-1, to=1-2]
	\arrow["\simeq", from=1-1, to=2-1]
	\arrow[from=1-2, to=1-3]
	\arrow["\simeq", from=1-2, to=2-2]
	\arrow[from=1-3, to=1-4]
	\arrow[from=1-3, to=2-3]
	\arrow[from=1-4, to=1-5]
	\arrow["\simeq", from=1-4, to=2-4]
	\arrow["\simeq", from=1-5, to=2-5]
	\arrow[from=2-1, to=2-2]
	\arrow[from=2-2, to=2-3]
	\arrow[from=2-3, to=2-4]
	\arrow[from=2-4, to=2-5]
\end{tikzcd}}\]
The rightmost vertical arrow is an isomorphism because $R\rightarrow R_{red}$ gives bijection between isomorphism classes of finitely generated projective $R$-module and finitely generated projective $R_{red}$-modules preserving the $C_2$-action. The leftmost vertical arrow and the second vertical arrow from the right are isomorphisms by  \ref{Ltheoryiso}. Therefore, it remains to show that the second vertical arrow from the left is an isomorphism. Once this is established, the theorem follows from the five lemma.
We now prove this claim.

Consider the homotopy orbit spectral sequence
$$E^2_{p,q} = H_p(C_2; K_q(R)) \implies \pi_{p+q}(K(R)_{hC_2})
$$ 
from \cite[Proposition B.2]{schlichting2017hermitian}. 
By naturality of the homotopy orbit spectral sequence with respect to the map of $C_2$-spectra 
$K(R)\longrightarrow K(R_{\mathrm{red}})$, 
we obtain a commutative diagram
\[\begin{tikzcd}[ampersand replacement=\&,cramped]
	{H_2(C_2, K_0(R))} \& {H_0(C_2, K_1(R))} \\
	{H_2(C_2, K_0(R_{red}))} \& {H_0(C_2, K_1(R_{red}))}
	\arrow["{d_2}", from=1-1, to=1-2]
	\arrow["\simeq", from=1-1, to=2-1]
	\arrow["\gamma", from=1-2, to=2-2]
	\arrow["{d_2}", from=2-1, to=2-2]
\end{tikzcd}\]
Since the canonical map $K_0(R) \rightarrow K_0(R_{\mathrm{red}})$ is an isomorphism, the left vertical map is an isomorphism.
Let $I$ be the nilradical of $R$. We note that $1+I$ is uniquely $2$-divisible group.  In fact, for $x\in I$ we have $(1+x)^{1/2}=1+\frac{1}{2}x-\frac{1}{8}x^2+\cdots$.
Consider the exact sequence of abelian groups:
\[
1 \longrightarrow 1+I \longrightarrow K_1(R) \longrightarrow K_1(R_{\text{red}}) \longrightarrow 1.
\]
The long exact sequence of group homology gives the following exact sequence.
\[\begin{tikzcd}[ampersand replacement=\&,cramped]
	{} \& {H_0(C_2,1+I)} \& {H_0(C_2,K_1(R))} \& {H_0(C_2,K_1(R_{red}))} \& 0
	\arrow[from=1-2, to=1-3]
	\arrow["\gamma", from=1-3, to=1-4]
	\arrow[from=1-4, to=1-5]
\end{tikzcd}\]
As $1+I$ is $2$-divisible and $2$ is invertible in $R$ it follows that $H_0(C_2, 1+I)=0$ and $\gamma$ is an isomorphism. Therefore, the induced map
$\alpha: H_0(C_2,K_1(R))/\operatorname{im}(d_2)
\rightarrow
H_0(C_2,K_1(R_{red}))/\operatorname{im}(d_2)$
is an isomorphism. Now consider the following commutative diagram with exact rows.
\[\begin{tikzcd}[ampersand replacement=\&,cramped]
	0 \& {H_0(C_2, K_1(R))/\text{im}(d_2)} \& {\pi_1 (K(R)_{hC_2})} \& { H_1(C_2,K_0(R))} \& 0 \\
	0 \& {H_0(C_2, K_1(R_{red}))/\text{im}(d_2)} \& {\pi_1 (K(R_{red})_{hC_2})} \& { H_1(C_2,K_0(R_{red}))} \& 0
	\arrow[from=1-1, to=1-2]
	\arrow[from=1-2, to=1-3]
	\arrow["\alpha", from=1-2, to=2-2]
	\arrow[from=1-3, to=1-4]
	\arrow[from=1-3, to=2-3]
	\arrow[from=1-4, to=1-5]
	\arrow["\simeq", from=1-4, to=2-4]
	\arrow[from=2-1, to=2-2]
	\arrow[from=2-2, to=2-3]
	\arrow[from=2-3, to=2-4]
	\arrow[from=2-4, to=2-5]
\end{tikzcd}\]
Since the canonical map $K_0(R) \rightarrow K_0(R_{\mathrm{red}})$ is an isomorphism, the right vertical arrow is an isomorphism. Consequently, the middle arrow is an isomorphism which proves the proposition.
\end{proof}
We improve \ref{reducedgw1} when the ring is the truncated polynomial ring. We shall need the following result of Weibel \cite[Consequence 1.4]{weibel2006mayer}.
\begin{prop}\label{weibelrelative}
    Let $I$ be a nilpotent ideal in a $\mathbb{Z}[\frac{1}{p}]$-algebra $R$. Then $K_*(R,I)$ is a $\mathbb{Z}[\frac{1}{p}]$-module.
\end{prop}

    We note that the map $\epsilon:\frac{R[X]}{(X^n)}\rightarrow R$ given by sending $X$ to $0$ induce surjection $K_*(\epsilon):K_*(\frac{R[X]}{(X^n)})\rightarrow K_*(R)$. Because we have inclusion $j: R\hookrightarrow \frac{R[X]}{(X^n)}$ and $\epsilon \circ j=id_R$.
    \begin{prop}
    Let $R$ be a ring with $1/2\in R$. Then the induced map $GW^{[n]}_i(\frac{R[X]}{(X^n)})\rightarrow GW^{[n]}_i(R)$ is an isomorphism for all $i$ and $n$.
\end{prop}
\begin{proof}
   Consider the following commutative diagram from \cite[Theorem 7.6]{schlichting2017hermitian}.
   \[\scalebox{0.8}{\begin{tikzcd}[ampersand replacement=\&,cramped]
	{L_{i+1}^{[n]}(\frac{R[X]}{(X^n)})} \& {\pi_{i+1}(K(\frac{R[X]}{(X^n)})_{hC_2})} \& {GW_{i+1}^{[n]}(\frac{R[X]}{(X^n)})} \& {L_{i}^{[n]}(\frac{R[X]}{(X^n)})} \& {\pi_{i}(K(R)_{hC_2})} \\
	{L_{i+1}^{[n]}(R)} \& {\pi_{i+1}(K(R)_{hC_2})} \& {GW_{i+1}^{[n]}(R)} \& {L_{i}^{[n]}(R)} \& {\pi_{i}(K(R)_{hC_2})}
	\arrow[from=1-1, to=1-2]
	\arrow["\simeq", from=1-1, to=2-1]
	\arrow[from=1-2, to=1-3]
	\arrow[from=1-2, to=2-2]
	\arrow[from=1-3, to=1-4]
	\arrow[from=1-3, to=2-3]
	\arrow[from=1-4, to=1-5]
	\arrow["\simeq", from=1-4, to=2-4]
	\arrow[from=1-5, to=2-5]
	\arrow[from=2-1, to=2-2]
	\arrow[from=2-2, to=2-3]
	\arrow[from=2-3, to=2-4]
	\arrow[from=2-4, to=2-5]
\end{tikzcd}}\]
As in the proof of \ref{reducedgw1} using \ref{weibelrelative} one can show that the rightmost vertical arrow and the second vertical arrow from the right are isomorphisms. The statement follows from five lemma.
\end{proof}
\begin{thm}\label{main1}
      Let $R\subset S$ be a finite subintegral extension. Let $1/2\in R$. Then the induced natural map $GW^{[n]}_0(R)\xrightarrow[]{}GW^{[n]}_0(S)$ is an isomorphism for $n=0,2$.
\end{thm}
  \begin{proof}
  By \ref{etaelementary}, we may assume that the extension $R\subset S$ is elementary subintegral extension. 
Let $C$ be the conductor ideal of the extension $R\subset S=R[b]$ with $b^2, b^3 \in R$. Then we have the following cartesian diagram of rings where $\epsilon: S/C=R/C\oplus \overline{b}R/C\twoheadrightarrow R/C$ denote the projection onto the first component. 
   \[\begin{tikzcd}[ampersand replacement=\&]\label{cartesian}
	R \& S \\
	{R/C} \& {S/C} \\
	{(R/C)_{red}} \& {(S/C)_{red}}
	\arrow["i", hook, from=1-1, to=1-2]
	\arrow["{\pi_1}"', from=1-1, to=2-1]
	\arrow["{\pi_2}", from=1-2, to=2-2]
	\arrow["j"', hook, from=2-1, to=2-2]
	\arrow[from=2-1, to=3-1]
	\arrow["\epsilon", shift left=3, curve={height=-6pt}, from=2-2, to=2-1]
	\arrow[from=2-2, to=3-2]
	\arrow[equals, from=3-1, to=3-2]
\end{tikzcd}\]
Consider the Mayer-Vietoris sequence \cite[Theorem. XII.8.3]{zbMATH03278303} of the above cartesian square.
\[\begin{tikzcd}[ampersand replacement=\&,row sep=small]
	{GW_1^{[n]}(S)\oplus GW_1^{[n]}(R/C)} \& {GW_1^{[n]}(S/C)} \& {GW_0^{[n]}(R)} \\
	{GW_0^{[n]}(S)\oplus GW_0^{[n]}(R/C)} \& {GW_0^{[n]}(S/C)} \& 0
	\arrow["\varphi", from=1-1, to=1-2]
	\arrow[from=1-2, to=1-3]
	\arrow[from=1-3, to=2-1]
	\arrow["\psi", from=2-1, to=2-2]
	\arrow[from=2-2, to=2-3]
\end{tikzcd}\]
The map $\varphi$ is surjective by \ref{reducedgw1} and $\psi$ is surjective because $(R/C)_{red}=(S/C)_{red}$. Therefore we have the following map of short exact sequences.
\[
\begin{tikzcd}[column sep=small, ampersand replacement=\&]
0 \& {GW_0^{[n]}(R)} 
\& {\begin{array}{c}GW_0^{[n]}(S)\\ \oplus \\ GW_0^{[n]}(R/C)\end{array}}
\& {GW_0^{[n]}(S/C)} \& 0 \\
0 \& {\begin{array}{c}
GW_0^{[n]}(S)\\
\simeq GW_0^{[n]}(S\times_{S/C}S/C)
\end{array}}
\& {\begin{array}{c}GW_0^{[n]}(S) \\ \oplus \\ GW_0^{[n]}(S/C)\end{array}} 
\& {GW_0^{[n]}(S/C)} \& 0
\arrow[from=1-1, to=1-2]
\arrow[from=1-2, to=1-3]
\arrow["i", from=1-2, to=2-2]
\arrow[from=1-3, to=1-4]
\arrow["id_S\oplus j", from=1-3, to=2-3]
\arrow[from=1-4, to=1-5]
\arrow[equals, from=1-4, to=2-4]
\arrow[from=2-1, to=2-2]
\arrow[from=2-2, to=2-3]
\arrow[from=2-3, to=2-4]
\arrow[from=2-4, to=2-5]
\end{tikzcd}
\]
The bottom exact sequence is from the Mayer-Vietoris sequence of the below cartesian square.
\[\begin{tikzcd}[ampersand replacement=\&]
	S \& {S} \\
	S/C \& {S/C}
	\arrow[from=1-1, to=1-2]
	\arrow[from=1-1, to=2-1]
	\arrow[from=1-2, to=2-2]
	\arrow[from=2-1, to=2-2]
\end{tikzcd}\] 
It follows that the natural map $GW_0^{[n]}(R)\rightarrow GW_0^{[n]}(S)$ is an isomorphism.
\end{proof}
For $n=2$, the above result recovers \cite[Corollary 3.20]{das2014invariance} and extends it to rings of arbitrary dimension.
\noindent 
We give an example where \ref{main1} doesn't hold for odd shifts.
\begin{ex}
    Consider the subintegral extension $R=\mathbb{C}[X^2,X^3]\subset \mathbb{C}[X]=S$. We know that $GW^{[1]}_0(S)=GW^{[1]}_0(\mathbb{C})=0$ and $GW^{[3]}_0(\mathbb{C})=\mathbb{Z}/2\mathbb{Z}$. Let $n=0$ or $2$. We get the following commutative diagram from \cite[Theorem 6.1]{schlichting2017hermitian}. 
    \[\begin{tikzcd}[ampersand replacement=\&,cramped]
	{GW^{[n]}_0(R)} \& {K_0(R)} \& {GW^{[n+1]}_0(R)} \\
	{GW^{[n]}_0(S)} \& {K_0(S)} \& {GW^{[n+1]}_0(S)}
	\arrow[from=1-1, to=1-2]
	\arrow["\simeq", from=1-1, to=2-1]
	\arrow[from=1-2, to=1-3]
	\arrow["i", two heads, from=1-2, to=2-2]
	\arrow[from=1-3, to=2-3]
	\arrow[from=2-1, to=2-2]
	\arrow[from=2-2, to=2-3]
\end{tikzcd}\]

    By \cite[Corollary 1]{ischebeck1989subintegral}, kernel of $K_0(R)\rightarrow K_0(S)$ is in bijection with the kernel of the map $\operatorname{Pic}(R)\rightarrow \operatorname{Pic(S)}=1$. As $\operatorname{Pic(R)}=(\mathbb{C},+)$ and $\operatorname{Pic(S)}= 1$, it follows that the induced map $GW^{[n+1]}_0(R)\rightarrow GW^{[n+1]}_0(S)$ is not an isomorphism.
\end{ex}

\section{Cancellation of Quadratic spaces over geometric subrings}\label{section2}
We briefly recall some basic notions concerning quadratic spaces. A detailed exposition may be found in \cite[Chapter 7, p.~233]{lam2006serre}. Throughout, let $R$ be a ring in which $2$ is invertible.

A \emph{quadratic space} over $R$ is a pair $(Q,d_q)$ consisting of a finitely generated projective $R$-module $Q$ together with a symmetric isomorphism $d_q:Q\xrightarrow{\sim}Q^*=\operatorname{Hom}_R(Q,R),$
that is, the transpose
$d_q^t:Q\xrightarrow{\sim}Q^{**}\xrightarrow{\sim}Q^*$ coincides with $d_q$. The map $
q:Q\to R$,  $q(x)=\frac{1}{2}d_q(x)(x),$ is called the \emph{quadratic form} associated to $(Q,d_q)$ and the map $B_q: Q\times Q\rightarrow R$ given by $B_q(x,y)=d_q(x)(y)$ is called the \emph{associated bilinear form} of $d_q$. By abuse of notation, we often write $q$ for the quadratic space $(Q,d_q)$.
An \emph{isometry} $\varphi:q\to q'$ is an isomorphism $\varphi:Q\to Q'$ such that the diagram
\[
\begin{tikzcd}
	Q & {Q'} \\
	{Q^*} & {(Q')^*}
	\arrow["\varphi", from=1-1, to=1-2]
	\arrow["{d_q}"', from=1-1, to=2-1]
	\arrow["{d_{q'}}", from=1-2, to=2-2]
	\arrow["{\varphi^*}", from=2-2, to=2-1]
\end{tikzcd}
\]
commutes. The \emph{orthogonal sum} of quadratic spaces $q=(Q,d_q)$ and $q'=(Q',d_{q'})$ is defined by $
q\perp q'=(Q\oplus Q',\, d_q\oplus d_{q'}).$ Let $B_q$ denote the symmetric bilinear form associated to $q$. If $Q_0\subseteq Q$ is a submodule, its orthogonal complement is
$
Q_0^{\perp}=\{x\in Q\mid B_q(x,y)=0 \text{ for all } y\in Q_0\}
$. A \emph{hyperbolic space} is a quadratic space of the form
$
\ H(P)=
(P\oplus P^*,
\begin{pmatrix}
0 & 1\\
1 & 0
\end{pmatrix}
),
$
where $P$ is a projective $R$-module. We write $\mathbf h^n$ for $H(R^n)$.
We say that a quadratic space $q$ has \emph{Witt index} at least $n$ if
$
q\simeq q_0\perp  H(P)
$
for some projective $R$-module $P$ of rank $n$. We say that $q$ has \emph{hyperbolic rank} at least $n$ if
$
q\simeq q_0\perp \mathbf h^d
$
for some $d\geq n$.
A quadratic space $q$ is said to be \emph{cancellative} if, whenever
$
q\perp q'\simeq q_1\perp q',
$
it follows that $q\simeq q_1$.

The following result is a consequence of \cite[Theorem 2]{wall1970classification}.

\begin{lem}\label{reduced}
Let $I\subset R$ be a nilpotent ideal, and let $q$ and $q'$ be quadratic spaces over $R$. If $q$ and $q'$ are isometric modulo $I$, then $q$ and $q'$ are isometric over $R$.
\end{lem}

 Roy introduced a class of orthogonal transformations called \emph{elementary orthogonal transformation}. We recall their definition and refer the reader to \cite{roy1968cancellation} for further details.

Let $(Q,q)$ be a quadratic space and let $P$ be a projective $R$-module. For a $R$-linear map
$\alpha\in \operatorname{Hom}_R(Q,P)$, $\beta\in \operatorname{Hom}_R(Q,P^*) $, let $\alpha^*=d_q^{-1}\circ \alpha^t$ and $\beta^*=d_q^{-1}\circ \beta^t$

define orthogonal transformations $E_\alpha$ and $E_\beta^*$ of $Q\perp \mathbb H(P)$ by
\[
\begin{aligned}
E_\alpha(z,x,f)
&=
\bigl(z-\alpha^*(f),
x+\alpha(z)-\tfrac12\alpha\alpha^*(f),
f\bigr),\\
E_\beta^*(z,x,f)
&=
\bigl(z-\beta^*(x),
x,
f+\beta(z)-\tfrac12\beta\beta^*(x)\bigr).
\end{aligned}
\]

We denote by $EO_R(Q,P)$ the subgroup of $O_R(Q\perp \mathbb H(P)),$
generated by the transformations $E_\alpha$ and $E_\beta^*$.

\begin{defn}\label{geometricsubring}
Let $R$ be a ring of dimension $d\geq 1$.  A Noetherian ring $A$ is called geometric subring of $R[T]$ if $dim(A)=d+1$, $R\subset A\subset R[T]$, and there exists a non zero divisor $s$ of $R$ such that $A_s=R_s[T].$
\end{defn}
Examples of geometric subrings include rees algebras. 
The notion of generalized dimension function is due to Plumstead. We recall the definition. Let $R$ be a ring. Let $X\subset \operatorname{Spec}(R)$. For a function $d:X\rightarrow \mathbb{N}$ we define a partial ordering $\ll$ on $X$ as follows. For primes $\mathfrak{p}$ and $\mathfrak{q}$ in $X$, define $\mathfrak{p} \ll \mathfrak{q}$ if and only if $\mathfrak{p} \subset \mathfrak{q}$ and $d(\mathfrak{p})> d(\mathfrak{q}).$ 
    \begin{defn}
        A function $d:X\rightarrow \mathbb{N}$ is generalized dimension function if for any ideal $I\subset R$, $V(I)\cap X$ has only finitely many minimal elements with respect to partial ordering $\ll.$
    \end{defn} 
    The following example is similar to \cite[Example 4]
    {plumstead1983conjectures}.
    \begin{ex}\label{gdf}
        Let $R$ be a ring of dimension $d$. Let $s\in R$ be a non zero divisor on $R$ such that $dim(R_s)\leq d-1$. Then there exists a generalized dimension function on $\operatorname{Spec}(R)$ whose values are atmost $d-1$.
    \end{ex}
\begin{thm}
    Let $R$ be a ring of dimension $d\geq 1$. Let $A$ be a geometric subring of $R[T]$. Let  $q$ be a quadratic space over $A$ of Witt index $\geq d+1$. Then $q$ is cancellative.  
\end{thm}
\begin{proof}
    We may assume $A$ is reduced using \ref{reduced}. Let $\mathcal{J}_A= TR[T]\cap A$. As $ht(\mathcal{J}_A)\geq 1$ there exists $F\in \mathcal{J}_A$ a non zero divisor on $A$. Let $''bar''$ denote modulo $F$. Since Witt index of $q$ $\geq d+1$, we have $q=q_0\perp H(P)$, where $P$ is projective $A$ module of rank 
    $\geq d+1$.  It is enough to show that if $q'\perp <u>\simeq  q\perp <u>$, where $u$ is an unit of $A$ and $q'$ a quadratic space over $A$, then $q'\simeq q$. Here $<u>$ denotes the quadratic form on $A$ given by sending $x\in A$ to $ux^2$.  Let $\beta: q'\perp <u>\xrightarrow[]{\sim}q\perp <u>$ be an isometry. Let $\beta(0,1)=(z,a)$. Since $rank(\Bar{P})\geq dim(\Bar{A})+1$, by \cite[Corollary 6.4]{roy1968cancellation}, there exists $\Bar{\eta}\in EO_{\Bar{A}}(\Bar{q_0}, H(\Bar{P}))$ such that $\Bar{\eta}(\Bar{z},\Bar{a})=(0,1)$. Let $\eta\in EO_A(q_0,H(P)) $be a lift of $\Bar{\eta}$. Replacing $\beta$ by $\beta\circ\eta$ we may assume $\Bar{\beta}(0,1)=(0,1)$. 

    Let $S$ be the set of all non zero divisors on $R$. Then $S^{-1}A=S^{-1}R[T]$, therefore by Harder's theorem \cite[Theorem 3.13, p.~246]{lam2006serre} $q\otimes_A S^{-1}A$ and $q'\otimes_A S^{-1}A$ are extended from $S^{-1}R$. We have isometries $\varphi_1:q\otimes_A S^{-1}A\xrightarrow[]{\sim}\Bar{q}\otimes_{R}S^{-1}A$ and $\varphi_2:q'\otimes_A S^{-1}A\xrightarrow[]{\sim}\Bar{q'}\otimes_{R}S^{-1}A$. Replacing $\varphi_i$ with $(\Bar{\varphi_i}^{-1}\otimes_RS^{-1}A)\circ \varphi_i$, for $i=1,2$, we may assume $\Bar{\varphi_i}=1$, for $i=1,2$. Then $\psi_1=\varphi_1^{-1}\circ \beta \circ \varphi_2$ is an isometry such that $\Bar{\psi}_1=\Bar{\beta}$. There exists $s\in S$ such that $\psi_1$ is defined over $A_s$.

  Let $S'=1+sR$ and $B=S'^{-1}A$. By \ref{gdf}, there exists an generalised dimension function $d$ on $\operatorname{Spec}(B)$ such that $d(B)\leq d$. Let  $\beta(0,1)=(Fz, Fx, Ff, 1+F\nu)$, where $z \in Q, x\in P, f \in P^*,$ and $\nu\in A$. 
  We have $i: FP\hookrightarrow P$, $(Fz,Fx,i^t(Ff), 1+F\nu)\in q_0\perp H(FP)\perp <u>$ and $FP$ a projective $A$ module of rank $\geq d+1$. Applying \cite[Theorem 3.1]{parimala1984cancellation}, there exists $\eta=E_{\alpha_1}\circ\cdots\circ E_{\alpha_n}\in EO_B(q_0\perp <u>, H(FP))$ such that $\eta(Fz,Fx,Ff,1+F\nu)=(Fz_0, Fx_0, i^t(Ff), 1+F\nu_0) $ with $x_0$ unimodular in $P$.
  Let $\widetilde{\eta}=E_{\widetilde{\alpha}_1}\circ E_{\widetilde{\alpha}_2}\circ \cdots \circ E_{\widetilde{\alpha}_n}$ where $\widetilde{\alpha}_j=i\circ \alpha_j$. 
  Then by \cite[Lemma 1.1]{parimala1984cancellation} $\widetilde{\eta}(fz,Fx,Ff,1+F\nu)=(Fz_0,Fx_0,Ff,1+F\nu_0)$ with  $x_0$ unimodular in $P$ and $\Bar{\widetilde{\eta}}=1$.
  
  Let $\beta_1^*:P=Bx_0\oplus P_0\rightarrow Q$ be a homomorphism with $\beta_1^*(x_0)=z_0$ and $\beta_1=(\beta_1^*)^t\circ d_{q_0}$. Then we have 
  $E_{\beta_1}^*(Fz_0,Fx_0,Ff, 1+F\nu_0)=(Fz_0-\beta_1^*(Fx_0),Fx_0, \beta_1(Fz_0)-1/2\beta_1\beta_1^*(Fx_0)+Ff,1+F\nu_0)$. As $\beta_1^*(Fx_0)=Fz_0$, we have $E_{\beta_1}^*(Fz_0,Fx_0,Ff, 1+F\nu_0)=(0,Fx_0, Ff_0,1+F\nu_0)$ for some $f_0\in P^*$. 
  
  Since $F^2f_0(x_0)+(1+F\nu_0)^2u=u$, $\nu_0=F\nu_1$ for some $\nu_1\in B$. Let $\beta_2^*:P=Bx_0\oplus P_0\rightarrow B$ be a homomorphism defined by $\beta_2^*(x_0)=F\nu_1$ and $\beta_2^*(P_0)=0$. Let $\beta_2=(\beta_2^*)^t\circ d_{<u>}$. Then we have $E^*_{\beta_2}(0,Fx_0,Ff_0,1+F\nu_0)=(0,Fx_0, \beta_2(1+F\nu_0)-1/2\beta_2\beta_2^*(Fx_0)+Ff_0,1+F\nu_0-\beta_2^*(Fx_0))=(0,Fx_0,Ff_1,1)$ and $\Bar{E}_{\beta_2}=1.$

  Let $\alpha:B\rightarrow P$ be homomorphism given by $\alpha(1)=-Fx_0$. As $f_1(x_0)=0$ we have $E_{\alpha_1}(0,Fx_0,Ff_1,1)=(0,0,Ff_1,1)$. Let $\beta_3:B\rightarrow P^*$ be homomorphisms $\beta_3(1)=-Ff_1$. Then $E_{\beta_3}^*(0,0,Ff_1,1)=(0,0,0,1)$. We have $\psi_2(0,1)=(E_{\beta_2}^*)^{-1}\circ E_{\beta_3}^*\circ E_{\alpha}\circ E_{\beta_2}^* \circ E_{\beta_1}^*\circ \widetilde{\eta} \circ \beta(0,1)=(0,1)$ and $\Bar{\psi}_2=\Bar{\beta}$. 
\end{proof}
The following is immediate consequence of above theorem.
\begin{cor}
    Let $A$ be as above. Let $q$ be a quadratic space over $A$ with Witt index $\Bar{q}\geq d+1$. If $q$ is stably extended from $A$, then it is extended from $A$.
\end{cor}
\noindent
\section*{Acknowledgments}
The author acknowledges the financial support provided by the Institute Post-Doctoral Fellowship, Indian Institute of Technology Madras, during the course of this research.
 \bibliographystyle{amsalpha}
\bibliography{ref.bib}  

\end{document}